\documentclass[12pt]{article}

\usepackage{amsfonts}
\usepackage{amsmath, amsthm}
\usepackage{amssymb}
\usepackage{indentfirst}
\usepackage[left=1in,right=1in,top=1in,bottom=1.5in]{geometry} %

\usepackage[english]{babel}
\usepackage{latexsym}
\usepackage{xcolor}
\renewcommand{\le}{\leqslant}
\renewcommand{\ge}{\geqslant}

\newcommand{\R}{\mathbb R}

\newtheorem{theorem}{Theorem}
\newtheorem{lemma}{Lemma}
\newtheorem{prop}{Proposition}

\newtheorem{defi}{Definition}
\newtheorem{question}{Question}

\begin{document}
\title{The right acute angles problem?}
\author{Andrey Kupavskii\footnote{Moscow Institute of Physics and Technology, IAS Princeton ; Email: {\tt kupavskii@ya.ru} \ \ Research supported by the grant of the Russian Government N 075-15-2019-1926.}, Dmitriy Zakharov\footnote{Higher School of Economics, Email: {\tt s18b1\_zakharov@179.ru}}}
\date{}
\maketitle

\begin{abstract}
The Danzer--Gr\"unbaum acute angles problem asks for the largest size of a set of points in $\R^d$ that determines only acute angles. There has been a lot of progress recently due to the results of the second author and of Gerencs\'er and Harangi, and now the problem is essentially solved.

 In this note, we suggest the following variant of the problem, which is one way to ``save'' the problem. Let $F(\alpha) = \lim_{d\to \infty} f(d,\alpha)^{1/d}$, where $f(d,\alpha)$ is the largest number of points in $\R^d$ with no angle greater than or equal to $\alpha$.  Then the question is to find $c:= \lim_{\alpha\to \pi/2^-}F(\alpha).$   It is an intriguing question whether $c$ is equal to $2$ as one may expect in view of the result of  Gerencs\'er and Harangi. In this paper we prove the  lower bound $c\ge \sqrt 2$.

 We also solve a related problem of Erd\H os and F\"uredi on the ``stability'' of the acute angles problem and refute another conjecture stated in the same paper.

\end{abstract}

\section{Introduction}

A set of points $X \subset \R^d$ is called {\it acute} ({\it non-obtuse}) if any three points from $X$ form an acute (acute or right, respectively) triangle. In 1962, Danzer and Gr\"unbaum  \cite{DG} confirmed a conjecture of Erd\H os from 1957 that any non-obtuse set of points in $\R^d$  has cardinality at most $2^d$, moreover, the only examples of non-obtuse sets of cardinality $2^d$ are the hypercube and some of its affine images. They then modified the question and asked to determine the maximum size $f(d)$ of an acute set in $\R^d$ for any $d \ge 2$. Danzer and Gr{\" u}nbaum  obtained the first bounds on $f(d)$:
\begin{equation} \label{dg}
2d-1 \le f(d) \le 2^d-1,
\end{equation}
where the upper bound immediately follows from the aforementioned result on non-obtuse sets. They conjectured that the lower bound is tight.

As it turned out recently, the value of $f(d)$ is actually very close to the upper bound in \eqref{dg}. While the only improvement upon the upper bound in \eqref{dg} made so far is the inequality $f(3) \le 5$ proved in \cite{C},  there were numerous improvements for the lower bound. The only values of $f(d)$ that are known at the moment are $f(2) = 3$ and  $f(3) = 5$, and the latter is the only known improvement of the upper bound (\ref{dg}), due to Croft
\cite{C}. %and this is the only non-trivial exact value of $f(d)$ known so far (along with the trivial equality $f(2) = 3$).

In 1983, Erd{\H o}s and F{\" u}redi \cite{EF} provided a probabilistic construction of an acute set with $[\frac{1}{2} (\frac{2}{\sqrt{3}})^d]$ points, thus disproving the conjecture of Danzer and Gr\"unbaum. The underlying idea was to consider a random subset of the vertices of the hypercube $\{0, 1\}^d$ (see the next section for details). In the years  1983-2009, the improvements of the lower bound were very moderate: the constant $\frac{1}{2}$ in front of the exponential term  $(\frac{2}{\sqrt{3}})^d$ was improved in several steps, resulting in the inequality $f(d) \gtrsim 0.942 \cdot (\frac{2}{\sqrt{3}})^d$ \cite{B, Bu}. In 2009, Ackerman and Ben-Zwi \cite{AB} improved the Erd{\H os}--F{\" u}redi bound by a factor of $c\sqrt{d}$ using a certain general result concerning the independence numbers of sparse hypergraphs. In 2001, Harangi \cite{H} made the first exponential improvement: the constant  $\frac{2}{\sqrt{3}} \approx 1.155$ was replaced by $(\frac{144}{23})^{0.1} \approx 1.201$. Harangi's idea was to  consider random subsets of the set of the form $X_0^n \subset \R^{d_0 n}$, rather than $\{0, 1\}^d$, as it was done in the proof by Erd{\H o}s and F{\" u}redi. Here, $X_0 \subset \R^{d_0}$ is a low-dimensional acute set, which is typically constructed by hand or with the help of computer. For example, if one takes $X_0$ to be an acute triangle on the plane then one gets the bound $f(d) \gtrsim 1.158^d,$ which is slightly better than the Erd\H os--F\"uredi bound. Harangi used a $12$-point acute subset of $\R^5$ in his proof.

The next round of development was triggered in the spring of 2017, when the first explicit exponential acute sets were constructed by the second author \cite{Z}. The obtained bound on $f(d)$ was also much better than the previously known ones: $f(d) \ge F_{d+1} \approx 1.618^d$, where $F_d$ is the $d$-th Fibonacci number.\footnote{Here  $F_0=F_1=1$.} The proof used induction and certain slight perturbations of the point set to make the right angles in the arising product-type constructions acute. In the fall of 2017 Gerencs{\' e}r and Harangi \cite{GH} proved that \begin{equation}\label{eqgh} f(d) \ge 2^{d-1}+1.\end{equation} The proof was inspired by constructions of $9$-point and $17$-point acute sets in $\R^4$ and $\R^5$, respectively, made by an Ukranian mathematics enthusiast. The idea of Gerencs\' er and Harangi's bound is to carefully perturb  the vertices of the hypercube $\{0, 1\}^{d-1}$ using one extra dimension to get rid of all right angles. One extra point can then be added to the construction.

One common feature of all known explicit exponential-sized constructions is that the largest angle among the points is just barely smaller than $\frac{\pi}{2}$, and the constructions break down completely if we require the largest angle to be, say $\frac{\pi}2-0.001$.  On the other hand, as we shall see below, random constructions can be usually modified so that the largest angle would be separated from $\frac{\pi}{2}$. This suggests a certain interesting direction for research, but let us first introduce a couple of definitions.

%the following question: how large are the So it seems interesting to find explicit constructions of acute sets in which the largest angle is, say, at most $\frac{\pi}{2} - \delta$ for some small $\delta$. To make the question more precise let us introduce a more general quantity.

\begin{defi}
Denote by $f(d, \alpha)$ the size of the largest set of points in $\R^d$ with no three points forming an angle at least $\alpha$. Put
\begin{equation}
    F(\alpha) := \limsup_{d \rightarrow \infty} f(d, \alpha)^{1/d}.
\end{equation}
\end{defi}
Thus, for instance, $f(d) = f(d, \frac{\pi}{2})$, and the result of  Gerencs{\' e}r--Harangi now implies that $F(\frac{\pi}{2}) = 2$. In \cite{Kup}, the first author showed that $\lim_{\alpha\to \pi/2^+}f(d,\alpha) = 2^{d}$.

 Note that $f(d, \alpha)$ is meaningful only for $\alpha \in [\frac{\pi}{3}, \pi]$ since $f(d, \alpha) = 2$ for any $\alpha \le \frac{\pi}{3}$. Some further results about $f(d, \alpha)$ for $\alpha$ close to $\frac{\pi}{3}$ or to $\pi$ can be found in \cite{EF}. %We are interested in the asymptotic behavior of $f(d, \alpha)$ for $\alpha$ near $\frac{\pi}{2}$ so let us define a limit:

Results of Erd{\H o}s--F{\" u}redi \cite[Theorem 3.6]{EF} translate to the following:
\begin{align}
    F\big(\frac{\pi}{3}+\delta\big) \in \big[1+\delta^2, 1+4\delta\big].
\end{align}

%\textbf{I don't understand this: is it known, or you think that it's obvious? In [EF] they ask a question of this sort. The bounds that you write also look ridiculously big to me. BTW, look at Conjecture 2.13 in [EF].}
%Analogously to the proof of Theorem 4.3 from \cite{EF} it can be shown that for any $\alpha > \frac{\pi}{2}$ and $d$ sufficiently large
In the range $\alpha > \frac{\pi}{2}$ it turns out that $f(d, \alpha)$ grows surprisingly fast. The following result is essentially due to Erd{\H o}s--F{\"u}redi \cite[Theorem 4.3]{EF} but their formulation applies only to $\alpha$ close enough to $\pi$ (note that the condition that $n$ is sufficiently large is missing in the statement of \cite[Theorem 4.3]{EF}).  % formulation is inaccurate (their bounds are valid only for $\alpha$ close to $\pi$).

\begin{prop}\label{pr1}
For any $\alpha \in (\frac{\pi}{2}, \pi)$ there are constants $C, c > 1$ such that for all sufficiently large $d$
\begin{equation}
    2^{c^{d}} < f(d, \alpha) < 2^{C^{d}}.
\end{equation}
\end{prop}
Note that Proposition \ref{pr1} refutes Conjecture 2.13 from the very same paper \cite{EF}.
%It means that if $\alpha > \frac{\pi}{2}$ then $F(\alpha) = \infty$. To prove Proposition \ref{pr1} we note that for any $\alpha = \frac{\pi}{2} + \varepsilon > \frac{\pi}{2}$ and sufficiently large $d$ there are at least $c^{d-1}$  vectors in $\R^d$ so that the angle between any two of them lies in $(\frac{\pi}{2}-\varepsilon, \frac{\pi}{2}+\varepsilon)$ (a random set on the sphere or in the $\{-1, 1\}^d$ cube will satisfy this condition). Then we can repeat the proof of Theorem 4.3 from \cite{EF}.

%Note that the statement of Theorem 4.3 from \cite{EF} is not accurate: for instance, it gives wrong results for all $n \le 2^{2^{d-1}}$. But the conclusion is valid (is it?) for all sufficiently large $n$. Also we note that Proposition \ref{pr1} refutes Conjecture 2.13 from \cite{EF}.

Now we can formulate our main question.

\begin{question}
Is it true that
\begin{equation}
    \lim_{\alpha \rightarrow \pi/2^-} F(\alpha) = 2?
\end{equation}
Equivalently, is it true that for any $\varepsilon > 0$ there is $\delta > 0$ so that for any sufficiently large $d$ there is a set $X \subset \R^d$ of cardinality at least $(2 - \varepsilon)^d$  such that any three points from $X$ determine an angle less than $\frac{\pi}{2} - \delta$? \end{question}

Although the problem is very close to the acute angles problem, the current methods that use explicit constructions fail completely, and the gap between the bounds is still exponential. We  prove the following lower bound in this paper.

\begin{theorem} \label{t1} We have
\begin{equation}
    \lim_{\alpha \rightarrow \pi/2^-} F(\alpha) \ge \sqrt{2}.
\end{equation}
 That is for every $\varepsilon > 0$ there exists $
 \delta>0$ such that for any sufficiently large $d$ there is a set $X \subset \R^d$ of cardinality at least $(\sqrt{2} - \varepsilon)^d$ determining only angles less than $\frac{\pi}{2} - \delta$.
\end{theorem}
Our proof is a combination of the method of Erd{\H o}s--F{\" u}redi with the recent construction of acute sets by Gerencs{\' e}r--Harangi.

The second result gives a non-trivial upper bound on $F(\alpha)$ for any $\alpha<\pi/2$.
\begin{theorem} \label{t2}
For $\alpha>0$ small enough we have $F(\frac{\pi}{2} - \alpha) \le 2 - \alpha^2$.
\end{theorem}

Theorem \ref{t2} confirms a conjecture of Erd{\H o}s--F{\" u}redi  \cite[Conjecture 3.5]{EF}. The proof is a  modification of the proof of the inequality $f(d) \le 2^d$ due to Danzer and Gr{\" u}nbaum. Namely, their proof is based on the observation that if $X$ is an acute set and $P = {\rm conv}(X)$ is the convex hull of $X$ then interiors of homothets $\frac{P + x}{2}$, $x \in X$, are pairwise disjoint. Considering the volumes one easily obtains the bound $|X| \le 2^d$.
The idea behind the proof of Theorem \ref{t2} is to take two disjoint subsets $A, C \subset X$ and consider sets of the form $\lambda \,{\rm conv} (A) + (1-\lambda) c \subset {\rm conv}(A \cup C)$, where $c \in C$. One can show that these sets are pairwise disjoint provided (i) all the angles in $X$ are less than $\frac{\pi}{2}-\alpha$ and (ii) $\lambda$ is chosen appropriately. One then obtains an inequality $\lambda^d {\rm Vol}({\rm conv\,} A) |C| \le {\rm Vol}({\rm conv\,}A \cup C)$. Lemma \ref{lm} implies that one can choose $A$ and $C$ in such a way that ${\rm Vol}({\rm conv\,}A)$ and ${\rm Vol}({\rm conv\,}A \cup C)$ are almost the same and $|C|$ is comparable to $|X|$, which completes the proof.

\section{The proofs}
\begin{proof}[Sketch of the proof of Proposition~\ref{pr1}] To prove the lower bound, we construct a set $\{v_1, \ldots, v_m\}$ of $m \ge c^d$ unit vectors in $\R^d$ such that the angle between any two of them lies in $(\frac{\pi}{2}-\varepsilon, \frac{\pi}{2}+\varepsilon)$, where $2\varepsilon = \alpha - \frac{\pi}{2}$. This can be done by taking a random subset on the unit sphere and applying a concentration inequality (see, for instance, \cite[Chapter 14]{M}). Now  take a sufficiently large number $\lambda$ and consider the set $X = \{ v_I = \sum_{t \in I} \lambda^t v_t~|~ I \subset [m]\}$. Note that $|X| = 2^{c^d}$. For any two points $v_I, v_J \in X$ we have $v_I - v_J \approx \pm \lambda^t v_t$, where $t$ is the largest element of $I \Delta J$. So the angle between $v_I-v_J$ and $v_I- v_K$ is approximately equal to the angle between some vectors $\pm v_i$ and $\pm v_j$, and therefore, it is at most $\alpha$.

To prove the upper bound, we construct a set $\{v_1, \ldots, v_m\}$ of $m \le C^d$ vectors such that any vector determines an angle less than $\frac{\pi - \alpha}{2}$ with one of them. This can be done by a greedy algorithm or deduced from known results for the sphere packing problem. Take a set $X$ of more than $2^{m}$ points. For $x,y\in X$, color a pair $(x, y), x\neq y$, in color $i$ if the angle between $v_i$ and $x-y$ is at most $\frac{\pi - \alpha}{2}$.
In what follows, we show that, since $|X| > 2^m$, there exists a triple $x, y, z$ such that $(x, y)$ and $(y, z)$ received the same color (i.e., there is a monochromatic oriented $2$-path). But then the angle between $y-x$ and $y-z$ is at least $\alpha$.

We show that such a triple exists by induction on $m$. The statement is clear for $m=1$ and $|X| = 3$. Next, for $m$-colorings, take any color, say, red, and consider all edges of this color. If there is no red oriented $2$-path, then each vertex either has only incoming or only outgoing red edges, and so red edges span a bipartite graph. (We are free to assign vertices with no incident red edge to any of the two parts.) Take the bigger part of this bipartite graph. It has size at least $\lceil (2^m+1)/2\rceil = 2^{m-1}+1$ and is colored with $m-1$ colors. Thus it contains a monochromatic oriented $2$-path.
\end{proof}

\begin{proof}[Proof of Theorem 1]
Fix an arbitrary $\varepsilon > 0$. Take a sufficiently large $d_0$ and an acute set $X_0 \subset \R^{d_0}$ of size $2^{d_0-1}+1$ (which exists by \eqref{eqgh}). Let $R > 0$ be the diameter of $X_0$ and denote by $s$ the smallest scalar product $\langle x - y, x - z\rangle$ over all triples $x, y, z \in X_0$ such that $x \neq y, z$. By the definition of an acute set, we have $s > 0$.

W.l.o.g., assume that $d_0$ divides $d$. Let $m = 2^{\frac{1-\varepsilon}{2} nd_0}$ where $n = d/d_0$. Choose $2m$ uniformly random points $p_1, \ldots, p_{2m} \in X_0^n \subset \R^{d_0n}$, and set $p_i = (p_{i1}, \ldots, p_{in})$. Let us estimate the expectation of the number of triples $(i, j, k)$ such that $\langle p_i - p_j, p_i - p_k\rangle \le \frac{\varepsilon}{2}  n s$.

If for some $i, j, k$ we have $\langle p_i - p_j, p_i - p_k\rangle \le \frac{\varepsilon}{2} n s$ then there are at least $(1-\frac{\varepsilon}{2})n$ coordinates $t \in \{1, \ldots, n\} $ for which  $p_{i t} = p_{j t}$ or $p_{i t} = p_{k t}$. The probability of the latter event is at most ${n \choose \frac{\varepsilon}{2}n} \left ( \frac{2}{|X_0|}  \right)^{(1-\frac{\varepsilon}{2})n} \le 2^{n - (1-\frac{\varepsilon}{2})(d_0 - 2)n}$. So the expectation of the number of such triples is at most
\begin{equation}
     (2m)^32^{n - (1-\frac{\varepsilon}2)(d_0 - 2)n} \le 8m 2^{(1 - \varepsilon)nd_0} 2^{-(1-\frac{\varepsilon}{2})nd_0+3n} \ll m.
\end{equation}
Thus there are points $p_1, \ldots, p_{2m}$ with at most $m$ ``bad'' triples. Remove one point from each of these triples and obtain a set $X \subset X_0^n \subset \R^{nd_0}$ of cardinality at least $m = \sqrt{2}^{(1-\varepsilon)nd_0}$ such that for any two points $x, y \in X$ we have $|x-y|^2 \le R^2 n$ and for any three points $x, y, z \in X$ we have $\langle x - y, x- z\rangle > \frac{\varepsilon}{2}ns$. This means that the angle $\alpha$ between vectors $x-y, x-z$ satisfies $\cos{\alpha} \ge \frac{\varepsilon}{2}s/R^2$ and thus depends on $\varepsilon$ only.
\end{proof}

In the proof of Theorem~\ref{t2}, we shall need the following lemma.
\begin{lemma}\label{lm}
Suppose $X \subset \R^d$, $|X| = N \ge d+1$ and the convex hull ${\rm conv}(X)$ has non-zero volume. Then for any $c \in [\frac{12d \log_2{N}}{N}, 1]$ there are sets $A \subset B \subset X$ such that \\
1. $|B \setminus A| \ge \frac{c}{3d \log_2{N}} N$.\\
2. $0 \neq {\rm Vol}({\rm conv}(B)) \le (1+c){\rm Vol}({\rm conv}(A)) $.

\end{lemma}
\begin{proof}
By Carath\'eodory's theorem, every point of ${\rm conv}(X)$ lies in the convex hull of some $d+1$ points of $X$, so by the pigeonhole principle, there is a set $X_0 \subset X$ of size $d+1$  such that
$$
{\rm Vol}({\rm conv}(X_0)) \ge {N \choose d+1}^{-1} {\rm Vol}({\rm conv}(X)) \ge N^{-d-1} {\rm Vol}({\rm conv}(X)).
$$
Take any chain $X_0 \subset X_1 \subset \ldots \subset X_m = X,$ such that $|X_{i+1} \setminus X_i| \in [\frac{c}{3d \log_2{N}} N, \frac{c}{2d \log_2{N}} N]$ (it is possible because of the restriction on $c$). We have $m \ge \frac{2d \log_2{N}}{c}$, so if we had ${\rm Vol}({\rm conv}(X_{i+1})) > (1+c){\rm Vol}({\rm conv}(X_i))$ for all $i$, then
$$
{\rm Vol}({\rm conv}(X)) > (1+c)^m {\rm Vol}({\rm conv}(X_0)) \ge 2^{2d \log_2{N}}{\rm Vol}({\rm conv}(X_0)) \ge {\rm Vol}({\rm conv}(X)),
$$
a contradiction.
\end{proof}

\begin{proof}[Proof of Theorem 2]
Take a set $X \subset \R^d$ which determines only angles at most $\frac{\pi}{2}- \alpha$ for a sufficiently small $\alpha>0$. Put $\varepsilon = \sin{\alpha}$. It is easy to see that for any three different points $x, y, z \in X$
\begin{equation}\label{ineq}
\langle y-x, z-x \rangle \ge \varepsilon\|y-x\|\|z-x\|> 1.5\varepsilon \|z-x\|^2,
\end{equation}
where the last inequality follows from the fact that $\frac{\|y-x\|}{\|z-x\|} = \frac{\sin\angle xzy}{\sin \angle zyx}> \sin \angle xzy\ge \sin 2\alpha >1.5\varepsilon$ for sufficiently small $\alpha$. Doing the same calculation for both $z-x$ and $x-z$ as the second vector in the scalar product in
\eqref{ineq}, we get that for any three distinct $ x,y,z$ we have
\begin{equation}\label{ineq2}
1.5 \varepsilon^2 \|z-x\|^2<\langle y-x, z-x \rangle < (1-1.5 \varepsilon^2) \|z-x\|^2.
\end{equation}

Applying Lemma 1 with $c = 1$ we get sets $A \subset B$ such that $0 \neq {\rm Vol}({\rm conv}B) \le 2{\rm Vol}({\rm conv}A)$ and $|B \setminus A| \ge \frac{|X|}{4d^2}$. Take $\lambda = \frac{1}{2} \cdot \left(1- 1.5\varepsilon^2\right)^{-1}$, from (\ref{ineq2}) we see that for any distinct $x, z \in B \setminus A$ we have $((1-\lambda)x+ {\rm conv} (\lambda A)) \cap ((1-\lambda)z+{\rm conv}(\lambda A)) = \emptyset$. Indeed, for any point $y$ from the first set we have $\langle y-x,z-x\rangle < \lambda (1-1.5\varepsilon^2)\|z-x\|^2 =  \frac 12\|z-x\|^2$, while for any $y'$ from the second set we have $\langle y'-x,z-x\rangle > (1-\lambda)\|z-x\|^2+\lambda\cdot 1.5 \varepsilon^2 \|z-x\|^2=\frac 12 \|z-x\|^2$.
Moreover, $(1-\lambda)x+ {\rm conv}(\lambda A)\subset {\rm conv} B$ for any $x\in B$, so
\begin{equation}
    |B\setminus A|\lambda^d {\rm Vol}({\rm conv}A) \le {\rm Vol}({\rm conv}B) \le 2 {\rm Vol}({\rm conv}A),
\end{equation}
thus
\begin{equation}
    |X| \le 4d^2|B\setminus A|\le 8d^2\lambda^{-d}=8d^2 2^d \left(1- 1.5\varepsilon^2\right)^d\le (2-\alpha^2)^d,
\end{equation}
provided that $d$ is sufficiently large and $\alpha>0$ is sufficiently small. (Here we used that $\lim_{\alpha\to 0^+} \frac{\sin \alpha}{\alpha} = 1$.)
\end{proof}

{\sc Acknowledgements: } We thank the reviewers for carefully reading the manuscript and suggesting numerous changes that helped to improve the exposition.


\begin{thebibliography}{20}
\bibitem[AB]{AB} E. Ackerman and O. Ben-Zwi, {\it On sets of points that determine only acute angles}, European Journal of Combinatorics 30 (2009), N4, 908--910.

\bibitem[B]{B} D. Bevan, {\it Sets of points determining only acute angles and some related colouring problems}, the electronic journal of combinatorics 13 (2006), N1, paper 12.

\bibitem[Bu]{Bu} L. V. Buchok, {\it Two New Approaches to Obtaining Estimates in the Danzer--Gr\"unbaum Problem}, Math. Notes, 87 (2010), N4, 489--496.

\bibitem[C]{C} H. T. Croft, {\it On 6-Point Configurations in 3-Space,} Journal of the London Mathematical Society 1 (1961),
    N1, 289--306.

\bibitem[DG]{DG} L. Danzer and B. Gr{\" u}nbaum, {\it "{\" U}ber zwei Probleme bez{\" u}glich konvexer K{\" o}rper von P. Erd{\H o}s und von VL Klee},  Mathematische Zeitschrift 79 (1962), N1, 95--99.

\bibitem[EF]{EF} P. Erd{\H o}s and Z. F{\" u}redi, {\it The greatest angle among n points in the d-dimensional Euclidean space}, Annals of Discrete Mathematics 17 (1983), 275--283.

\bibitem[GH]{GH} B. Gerencs\'er and V. Harangi, {\it Acute sets of exponentially optimal size}, Discrete \& Computational Geometry 62 (2019), N4, 775--780.


\bibitem[H]{H} V. Harangi, {\it Acute sets in Euclidean spaces}, SIAM Journal on Discrete Mathematics 25 (2011), N3, 1212--1229.

\bibitem[M]{M} Matou{\v s}ek, Ji{\v r}{\' i}. {\it Lectures on discrete geometry}. Vol. 212. New York: Springer, 2002.

\bibitem[Kup]{Kup} A. Kupavskii, {\it Number of double-normal pairs in space}, Discrete and Computational Geometry 56 (2016), N3, 711--726.

\bibitem[Z]{Z} D. Zakharov, {\it Acute sets}, Discrete \& Computational Geometry 61 (2019), N1, 212--217.


\end{thebibliography}
\end{document}